\documentclass{amsart}

\usepackage[utf8]{inputenc}
\usepackage{lmodern}
\usepackage[T1]{fontenc}
\usepackage{amsfonts} 
\usepackage{mathtools}
\usepackage{amsthm}
\usepackage{amsmath}
\usepackage{enumerate}
\usepackage{amssymb}

\theoremstyle{definition}

\theoremstyle{plain}
\newtheorem{theorem}{Theorem}
\newtheorem{lemma}{Lemma}
\newtheorem{proposition}{Proposition}

\newtheorem{corollary}{Corollary}

\DeclareMathOperator{\vrt}{vert}
\DeclareMathOperator{\conv}{conv}
\DeclareMathOperator{\pyr}{pyr}
\DeclareMathOperator{\aff}{aff}
\DeclareMathOperator{\op}{op}

\newcommand{\bone}{\textbf{1}}

\usepackage[backend=bibtex]{biblatex}
\title{$f$-vector inequalities for order and chain polytopes}
\author{Ragnar Freij-Hollanti and Teemu Lundström}

\begin{document}
\maketitle

\begin{abstract}
    The order and chain polytopes are two 0/1-polytopes constructed from a finite poset.
    In this paper, we study the $f$-vectors of these polytopes.
    We investigate how the order and chain polytopes behave under disjoint unions and ordinal sums of posets, and how the $f$-vectors of these polytopes are expressed in terms of $f$-vectors of smaller polytopes.
    Our focus is on comparing the $f$-vectors of the order and chain polytope built from the same poset.
    In our main theorem we prove that for a family of posets built inductively by taking disjoint unions and ordinal sums of posets, for any poset $\mathcal{P}$ in this family the $f$-vector of the order polytope of $\mathcal{P}$ is component-wise at most the $f$-vector of the chain polytope of $\mathcal{P}$.
\end{abstract}

\section{Introduction}\label{section:intro}
The order polytope $\mathcal{O}(\mathcal{P})$ and chain polytope $\mathcal{C}(\mathcal{P})$ of a poset $\mathcal{P}$ were defined in 1986 by Stanley \cite{Stanley}, and are two of the most prominent examples of convex polytopes occurring naturally in algebraic combinatorics.
In particular, their toric rings are both examples of so-called algebras with straightening laws on distributive lattices \cite{Hibi87, Hibi&Li3}. 
The polytopes are well known to have the same Ehrhart polynomial, and in particular the same number of vertices and the same dimension, but in general not the same number of facets. 
Hibi and Li showed that indeed, the number of facets of the chain polytope is always at least the number of facets of the order polytope, and that equality holds if and only if $\mathcal{P}$ does not contain the so-called X-poset (Figure \ref{fig:Xposet}) as a subposet \cite{Hibi&Li}.
In the same paper, they conjectured that the $f$-vector of $\mathcal{C}(\mathcal{P})$ coordinate-wise dominates that of $\mathcal{O}(\mathcal{P})$ for any poset $\mathcal{P}$. 
The first step towards a proof came in \cite{edges}, when it was shown that the number $f_1$ of edges of the two polytopes agrees, for all posets. 
This proof used an entirely combinatorial interpretation of the edges of the two polytopes. 
While similar combinatorial descriptions for the higher-dimensional faces of $\mathcal{O}(\mathcal{P})$ occurred already in~\cite{Stanley}, such a description seems more evasive for the chain polytopes, yielding bijective proofs infeasible. 
Similar differences in the combinatorics occur in the recent generalizations that are the {\em marked} order and chain polytopes \cite{Ardila, Jochemko}. 
A generalization of Hibi's and Li's conjecture to marked order and chain polytopes was suggested by Fang and Fourier \cite{Fang}.

Recently, it was proved that the inequality conjectured by Hibi and Li holds for all so called maximal ranked posets \cite{Ahmad-Fourier-Joswig}.
In this paper, we prove the conjecture for a larger class of posets using geometric techniques, by considering a recursive structure on the order and chain polytopes. 
Our main result is that the inequality $f(\mathcal{C}(\mathcal{P}))\geq f(\mathcal{O}(\mathcal{P}))$ holds for the entire $f$-vector if $\mathcal{P}$ is recursively constructed from so-called $X$-free posets, using ordinal sums and disjoint unions. 
The rest of the paper is organized as follows. 
In Section~\ref{section:preliminaries}, we give the necessary definitions, which mostly follow the standards in the literature.
In Section~\ref{section:PvQ}, we go through the details of a special case of the subdirect sum, first defined by McMullen \cite{McMullen}.
We provide details to some of the proofs omitted in \cite{McMullen}.
In Section~\ref{section:O(A<B)_and_C(A<B)}, we analyze the geometry of the order and chain polytopes of ordinal sums, using this subdirect sum. 
In Section~\ref{section:f_polynomials}, we derive $f$-vector inequalities from the geometric considerations, which allows us to prove our main theorem in Section~\ref{section:main_theorem}.

\begin{figure}[t]
    \centering
    \includegraphics[scale=0.8]{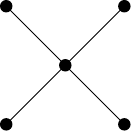}
    \caption{Hasse diagram for the 5-element poset ''$X$-poset''.}
    \label{fig:Xposet}
\end{figure}

\section{Preliminaries}\label{section:preliminaries}
We start by fixing notation and terminology.
If $E$ is any finite set then $\mathbb{R}^E$ denotes the real vector space of all sequences $(x_i)_{i \in E}$ with $x_i \in \mathbb{R}$ for all \mbox{$i \in E$}, together with point-wise addition and scalar multiplication.
In particular, $\mathbb{R}^{\emptyset}$ contains a single element, the empty sequence.
If $i \in E$ then $e_i$ is the vector of $\mathbb{R}^E$ with $1$ in the $i$-th coordinate and 0 elsewhere.
If $A \subseteq E$ we let $e_A = \sum_{i \in A}e_i$.
In general, if $S \subseteq \mathbb{R}$ then $S^E$ is the set of sequence $(x_i)_{i \in E}$ with $x_i \in S$ for all $i \in E$.
If $F$ and $E$ are disjoint finite sets and if $X \subseteq \mathbb{R}^E$ and $Y \subseteq \mathbb{R}^F$, then $X \times Y$ is the set of all sequences $(x_i)_{i \in E \cup F}$ where $(x_i)_{i \in E} \in X$ and $(x_i)_{i \in F} \in Y$.
The cardinality of a finite set $X$ is denoted by $\#X$.

All posets in this paper are assumed to be finite.
A subset of a poset is always identified with the corresponding induced subposet.
If $\mathcal{P}$ is a poset then $\mathcal{P}^{\op}$ is the poset on the same set where $x \le y$ in $\mathcal{P}^{\op}$ if $y \le x$ in $\mathcal{P}$.
Suppose $A$ and $B$ are posets on disjoint sets.
We let $A < B$ denote the \emph{ordinal sum} of $A$ and $B$, that is, $A < B$  is the poset on $A \cup B$ where $x \le y$ in  $A < B$ if $x \le y$ in $A$, or if $x \le y$ in $B$, or if $x \in A$ and $y \in B$.
The ordinal sum is sometimes denoted by $A \oplus B$, for example in \cite{Stanley_EC1}, but we use $A < B$ to emphasize that the operation is (in general) not commutative.
We also let $A \sqcup B$ denote the poset on $A \cup B$ where $x \le y$ in $A \sqcup B$ if $x \le y$  in $A$ or if $x \le y$ in $B$.
A \emph{filter} of a poset $\mathcal{P}$ is a subset $F \subseteq \mathcal{P}$ such that, whenever $x \in F$ and $y \ge x$ in $\mathcal{P}$ then $y \in F$.

All polytopes in this paper are convex.
For basic notions on convex polytopes we refer to \cite{Ziegler}.
The face lattice of a polytope $P$ is denoted by $L(P)$ and the vertex set of $P$ is denoted by $\vrt(P)$.
The empty set is considered to be a face of $P$ of dimension $-1$.
We also consider the entire polytope $P$ to be a face of $P$.
The join of two polytopes $P$ and $Q$ is denoted by $P*Q$ and the pyramid over $P$ is denoted by $\pyr(P)$.
If $v$ is a vertex of a polytope $P$ then the vertex figure of $P$ at $v$ is denoted by $P/v$.

Next, we quickly recall the definitions and basic properties of the order and the chain polytopes.
For more on these polytopes, see for example Stanley's paper \cite{Stanley}.

Let $\mathcal{P}$ be a finite poset.
The order polytope of $\mathcal{P}$ is the polytope $\mathcal{O}(\mathcal{P}) \subseteq \mathbb{R}^\mathcal{P}$ defined by the equations
\begin{align*}
    0 \le x_i \le 1, \quad &\text{for all } i \in \mathcal{P} \\
    x_i \le x_j, \quad &\text{for all } i \le j \text{ in } \mathcal{P}.
\end{align*}
The vertices of $\mathcal{O}(\mathcal{P})$ are the vectors $e_F$ where $F$ is a filter of $\mathcal{P}$ and the facets of $\mathcal{O}(\mathcal{P})$ are given by the equations 
\begin{align*}
    x_i = 0, \quad i  &\text{ minimal in }  \mathcal{P}\\
    x_j = 1, \quad j &\text{ maximal in } \mathcal{P} \\
    x_ i = x_j, \quad i &\text{ covered by } j \text{ in } \mathcal{P}.
\end{align*} 
In the literature, the order polytope is sometimes instead defined by requiring that $~{x_i \ge x_j}$ for all $i \le j$ in $\mathcal{P}$ \cite{Hibi&Li}.
The polytope defined this way is the polytope $\mathcal{O}(\mathcal{P}^{\op})$ in our definition, and the map
\begin{align*}
    \mathbb{R}^\mathcal{P} &\longrightarrow \mathbb{R}^\mathcal{P} \\
    x &\longmapsto \bone - x
\end{align*}
where $\bone \in \mathbb{R}^\mathcal{P}$ is the all 1's vector, gives an isomorphism $\mathcal{O}(\mathcal{P}) \to \mathcal{O}(\mathcal{P}^{\op})$.

The chain polytope of $\mathcal{P}$ is the polytope $\mathcal{C}(\mathcal{P}) \subseteq \mathbb{R}^\mathcal{P}$ defined by equations 
\begin{align*}
    x_i \ge 0, &\quad \text{for all } i \in \mathcal{P} \\
    \sum_{i \in C} x_i \le 1, &\quad \text{for all chains } C \text{ of }\mathcal{P}. 
\end{align*}
The vertices of $\mathcal{C}(\mathcal{P})$ are the vectors $e_A$ where $A$ is an antichain of $\mathcal{P}$, and the facets are given by the equations
\begin{align*}
    x_i &= 0, \quad i \in \mathcal{P}\\
    \sum_{i \in C} x_i &= 1, \quad C \text{ maximal chain in } \mathcal{P}.
\end{align*}
Both $\mathcal{O}(\mathcal{P})$ and $\mathcal{C}(\mathcal{P})$ are full dimensional polytopes, hence
\begin{equation*}
    \dim(\mathcal{O}(\mathcal{P})) = \dim (\mathcal{C}(\mathcal{P})) = \# \mathcal{P}.
\end{equation*}

The focus of this paper is on $f$-vectors.
However, instead of using $f$-vectors we will use $f$-polynomials, since this allows an easier manipulation of these $f$-vectors.
Here we define the $f$-polynomial of a polytope $P$ of dimension $d$ as the polynomial
\begin{equation*}
    f_P \coloneqq \sum_{i=-1}^d f_ix^{i+1} = \sum_{F \in L(P)} x^{\dim F + 1} \in \mathbb{N}[x]
\end{equation*}
where $f_i$ is the number of $i$-dimensional faces of $P$, for all $i=-1,0,1,\dots,d$.
Later, we will need to consider the number of faces containing the origin and not containing the origin.
For this, we introduce two additional polynomials.
If $P$ is a polytope containing the origin as a vertex, let
\begin{align*}
    f_P^0 &\coloneqq \sum_{\substack{F \in L(P) \\ 0 \in F}} x^{\dim(F) + 1} \\
    f_P^1 &\coloneqq \sum_{\substack{F \in L(P) \\ 0 \not\in F}} x^{\dim(F) + 1}.
\end{align*}
Note that $f_P = f_P^0 + f_P^1$.
Note also that since the $k$-dimensional faces of $P$ containing~$0$ are in bijection with the $(k-1)$-dimensional faces of the vertex figure $P/0$, we have $f_P^0 = x f_{P/0}$.

If $f, g \in \mathbb{N}[x]$ are any polynomials with non-negative integer coefficients, we write $f \le g$ to mean that the coefficient of $x^k$ in $f$ is at most the coefficient of $x^k$ in $g$, for all $k \ge 0$.
Hence, if $P$ and $Q$ are two polytopes with $f$-polynomials $f_P$ and $f_Q$, then $f_P \le f_Q$ if and only if the $f$-vector of $P$ is component-wise at most the $f$-vector of $Q$.
Note that in $\mathbb{N}[x]$, inequalities between polynomials are preserved under multiplication.
That is, if $f_1,f_2,g_1,g_2 \in \mathbb{N}[x]$ are such that $f_1 \le f_2$ and $g_1 \le g_2$, then $f_1g_1 \le f_2g_2$.

\section{Subdirect sum}\label{section:PvQ}
In this section we look at a special case of a polytope construction called the subdirect sum.
Our motivation for studying this construction comes from the following observation.

Let $A$ and $B$ be posets on disjoint sets.
Now in the ordinal sum $A < B$ we have
\begin{equation*}
    \{ \text{antichains of } A < B\} = \{ \text{antichains in } A \} \cup \{ \text{antichains in } B \}.
\end{equation*}
Therefore
\begin{equation*}
    \vrt(\mathcal{C}(A < B)) = \vrt(\mathcal{C}(A)) \times \{ 0 \}^B \cup \{ 0 \}^A \times \vrt(\mathcal{C}(B))
\end{equation*}
and hence
\begin{equation*}
    \mathcal{C}(A<B) = \conv(\mathcal{C}(A) \times \{ 0 \}^B \cup \{ 0 \}^A \times \mathcal{C}(B)).
\end{equation*}
Phrased geometrically, $\mathcal{C}(A<B)$ is constructed from $\mathcal{C}(A)$ and $\mathcal{C}(B)$ by gluing them together at the origin so that their affine hulls meet only in a single point, namely, at the origin.

In general, let $P \subseteq \mathbb{R}^m$ and $Q \subseteq \mathbb{R}^n$ be polytopes with $0 \in \vrt(P)$ and \mbox{$0 \in \vrt(Q)$}.
We define
\begin{equation*}
    P \vee Q \coloneqq \conv(P \times \{ 0 \}^n \cup \{ 0 \}^m \times Q) \subseteq \mathbb{R}^{m+n}.
\end{equation*}
This construction is a special case of the subdirect sum.
For more on the subdirect sum, see for example McMullen's paper \cite{McMullen}.
Using the notation of that paper, our construction here could be written as
\begin{equation*}
    P \vee Q = (P,\{ 0 \}) \oplus(Q,\{ 0 \}),
\end{equation*}
although we won't use that notation here.
In his paper \cite{McMullen}, McMullen gives a description of the faces of $P \vee Q$ in terms of the faces of $P$ and $Q$ (\cite{McMullen}, Proposition~2.3) although without a detailed proof.
We will fill in those details now.
For the rest of this section, we let $P$ and $Q$ be fixed polytopes as above.

First, we notice that the dimension behaves nicely.
\begin{proposition}
    $\dim(P \vee Q) = \dim(P) + \dim(Q)$.
\end{proposition}

\begin{proof}
    We have
    \begin{align*}
        \aff(P \vee Q) &= \aff(\conv(P \times \{ 0 \}^n \cup \{ 0 \}^m \times Q)) \\
        &= \aff(P \times \{ 0 \}^n \cup \{ 0 \}^m \times Q).
    \end{align*}
    Since both $P$ and $Q$ contain the origin, we have
    \begin{equation*}
        \{ 0 \} = \aff(P \times \{ 0 \}^n) \cap \aff(\{ 0 \}^m \times Q).
    \end{equation*}
    Thus
    \begin{align*}
        \dim(P \vee Q) &= \dim(P \times \{ 0 \}^n \cup \{ 0 \}^m \times Q) \\
        &= \dim(P \times \{ 0 \}^n) + \dim(\{ 0 \}^m \times Q) \\
        &= \dim(P) + \dim(Q).
    \end{align*}
\end{proof}

Next, we look at the vertices of $P \vee Q$.
From the definition we get
\begin{equation*}
    P \vee Q = \conv(\vrt(P) \times \{ 0 \}^n \cup \{ 0 \}^m \times \vrt(Q)).
\end{equation*}
Therefore $\vrt(P \vee Q) \subseteq \vrt(P) \times \{ 0 \}^n \cup \{ 0 \}^m \times \vrt(Q)$.
For the rest of this section, write $\vrt(P) \setminus \{ 0 \} = \{ v_1, \dots, v_k \}$ and $\vrt(Q) \setminus \{ 0 \} = \{ w_1,\dots,w_{\ell} \}$.

\begin{proposition}\label{prop:vertices_of_PvQ}
    The vertices of $P \vee Q$ are given by
    \begin{equation*}
        \vrt(P \vee Q) = \{ (v_1,0),\dots,(v_k,0), (0,w_1),\dots,(0,w_\ell), (0,0) \} \subseteq \mathbb{R}^{m+n}.
    \end{equation*}
\end{proposition}

\begin{proof}
    We saw the inclusion $\subseteq$ above.
    To prove the inclusion $\supseteq$ we need to show that no point in the set on the right-hand side is in the convex hull of the rest.
    We start with $(0,0)$.
    Suppose we had a convex combination
    \begin{equation*}
        (0,0) = \sum_{i=1}^k t_i(v_i,0) + \sum_{j = 1}^\ell s_j (0,w_j). \tag{*}
    \end{equation*}
    From this we obtain a linear combination $0 = t_1v_1 + \dots + t_kv_k$.
    Since $0$ is a vertex of $P$ we find some $\alpha \in (\mathbb{R}^m)^*$ such that $\alpha(v_i) < 0$ for all $i \in \{ 1,\dots,k \}$.
    Hence we have
    \begin{equation*}
        0 = \alpha(0) = t_1 \alpha(v_1) + \dots + t_k \alpha(v_k).
    \end{equation*}
    Here $t_i\alpha(v_i) \le 0$ for all $i$.
    Therefore $t_i \alpha(v_i) = 0$ for all $i$.
    Since $\alpha(v_i) < 0$, we need to have $t_i = 0$ for all $i = 1,\dots,k$.
    Now from (*) we see that we have a convex combination $0 = s_1w_1 + \dots + s_\ell w_\ell$, which is a contradiction as $0$ is a vertex of $Q$.
    Hence $(0,0)$ is not in the convex hull of the other points.
    If we had a convex combination, say,
    \begin{equation*}
        (v_1,0) = \sum_{i=2}^k t_i(v_i,0) + \sum_{j=1}^\ell s_j(0,w_j) + r(0,0)
    \end{equation*}
    then we would have a convex combination
    \begin{equation*}
        v_1 = \sum_{i=2}^k t_iv_i + (1-t_2 -\dots -t_k)0
    \end{equation*}
    and thus $v_1 \in \conv(v_2,\dots,v_k,0)$, a contradiction since $v_1$ is a vertex of $P$.
    Similarly for all the other points.
    We are done.
\end{proof}

Before we start looking at the faces of $P\vee Q$, we recall the following basic result:
The faces of a polytope are exactly those convex hulls of vertices that maximize some linear function.
More specifically, let $R \subseteq \mathbb{R}^d$ be any polytope.
Let $W \subseteq \vrt(R)$.
Then $\conv(W)$ is a face of $R$ if and only if there exists $\alpha \in (\mathbb{R}^d)^*$ and $c \in \mathbb{R}$ such that for all $v \in \vrt(R)$,
\begin{enumerate}[(i)]
    \item $\alpha(v) \le c$, and
    \item $\alpha(v) = c$ if and only if $v \in W$.
\end{enumerate}
If $W \subseteq \vrt(R)$ and there exists $\alpha \in (\mathbb{R}^d)^*$ and $c \in \mathbb{R}$ with the properties $(i)$ and $(ii)$ above, we say $W$ \emph{maximizes the equation} $\alpha(x) \le c$ \emph{for} $R$, which implies that $\conv(W)$ is a face of $R$.

\begin{lemma}\label{lemma:if_0_in_H_then_H=FvG}
    Let $H \in L(P \vee Q)$ with $0 \in \vrt(H)$.
    Then $H = F \vee G$ for some $F \in L(P)$ and $G \in L(Q)$ where $0 \in F$ and $0 \in G$.
\end{lemma}

\begin{proof}
    Write
    \begin{equation*}
        \vrt(H) = \{ (v_1,0),\dots,(v_s,0),(0,w_1),\dots,(0,w_r), (0,0) \}
    \end{equation*}
    where $v_i \in \vrt(P) \setminus \{ 0 \}$ and $w_i \in \vrt(Q) \setminus \{ 0 \}$.
    Let $F = \conv(v_1, \dots, v_s,0)$ and $G =  \conv(w_1, \dots, w_r,0)$.
    To prove the first claim the simply compute
    \begin{align*}
        F \vee G &= \conv(F \times \{ 0 \}^n \cup \{ 0 \}^m \times G) \\
        &= \conv(\conv(v_1, \dots, v_s,0) \times \{ 0 \}^n \cup \{ 0 \}^m \times \conv(w_1, \dots, w_r,0) ) \\
        &= \conv(\conv(v_1,0),\dots,(v_s,0),(0,0)) \cup \conv((0,w_1),\dots,(0,w_r),(0,0))) \\
        &= \conv((v_1,0),\dots,(v_s,0),(0,w_1),\dots,(0,w_r),(0,0)) \\
        &= H.
    \end{align*}
    Since $H$ is a face of $P \vee Q$ containing the origin, we find some $\gamma \in (\mathbb{R}^{m+n})^*$ such that $\vrt(H)$ maximizes the equation $\gamma(x,y) \le 0$ for $P \vee Q$.
    We can write $\gamma(x,y) = \alpha(x) + \beta(y)$ for some $\alpha \in (\mathbb{R}^{m})^*$ and $\beta \in (\mathbb{R}^n)^*$.
    One easily shows that $\vrt(F)$ maximizes the equation $\alpha(x) \le 0$ for $P$ and that $\vrt(G)$ maximizes the equation $\beta(y) \le 0$ for $Q$.
    Thus $F \in L(P)$ and $G \in L(Q)$.
\end{proof}

\begin{lemma}\label{lemma:if_0_in_F_and_0_in_G_then_FvG_in_L(PvQ)}
    If $F \in L(P)$ and $G \in L(Q)$ with $0 \in F$ and $0 \in G$, then
    \begin{equation*}
        F \vee G \in L(P \vee Q)
    \end{equation*}
    where $0 \in F \vee G$.
\end{lemma}

\begin{proof}
    Write
    \begin{equation*}
        \vrt(P) = \{ v_1,\dots,v_s,0 \} \sqcup \{ v_{s+1},\dots,v_k \}
    \end{equation*}
    where $\{ v_1, \dots, v_s ,0\} = \vrt(F)$, and write
    \begin{equation*}
        \vrt(Q) = \{ w_1,\dots,w_{r},0 \} \sqcup \{ w_{r+1},\dots,w_{\ell}\}
    \end{equation*}
    where $\vrt(G) = \{ w_1,\dots,w_r ,0\}$.
    Now by Proposition \ref{prop:vertices_of_PvQ} we have
    \begin{equation*}
        \vrt(F \vee G) = \{ (v_1,0) ,\dots, (v_s,0), (0,w_1),\dots,(0,w_r),(0,0) \}.
    \end{equation*}
    Since $F \in L(P)$ and $G \in L(Q)$ where both $F$ and $G$ contain the origin, we find $\alpha \in (\mathbb{R}^m)^*$ and $\beta \in (\mathbb{R}^n)^*$ such that $\vrt(F)$ maximizes $\alpha(x) \le 0$ for $P$ and $\vrt(G)$ maximizes $\beta(y) \le 0$ for $Q$.
    Let $\gamma \in (\mathbb{R}^{m+n})^*$ be defined by $\gamma(x,y) = \alpha(x) + \beta(y)$.
    Now one easily shows that $\vrt(F \vee G)$ maximizes $\gamma(x,y) \le 0$ for $P \vee Q$.
    Thus $F \vee G \in L(P \vee Q)$.
\end{proof}

For the rest of this section we let
\begin{align*}
    \varphi \colon \mathbb{R}^{m+n+1} &\longrightarrow \mathbb{R}^{m+n} \\
    (x,y,z) &\longmapsto (x,y).
\end{align*}
We note that if $F \subseteq \mathbb{R}^m$ and $G \subseteq \mathbb{R}^n$ are any polytopes with $\vrt(F) = \{ v_1,\dots,v_s \}$ and $\vrt(G) = \{ w_1,\dots,w_r \}$, then
\begin{equation*}
    F*G = \conv((v_1,0,0),\dots,(v_s,0,0),(0,w_1,1),\dots,(0,w_r,1)) \subseteq \mathbb{R}^{m+n+1}
\end{equation*}
and thus
\begin{equation*}
    \varphi(F*G) = \conv((v_1,0),\dots,(v_s,0),(0,w_1),\dots,(0,w_r)).
\end{equation*}

\begin{lemma}\label{lemma:if_0_not_in_H_then_H_isom_F*G}
    Let $H \in L(P \vee Q)$ and suppose $0 \not\in H$.
    Then $\varphi(F*G) = H$ for some $F \in L(P)$ and $G \in L(Q)$ where $0 \not\in F$ and $0 \not\in G$.
\end{lemma}

\begin{proof}
    Since $0 \not\in H$ we can write
    \begin{equation*}
        \vrt(H) = \{ (v_1,0),\dots,(v_s,0),(0,w_1),\dots,(0,w_r) \}
    \end{equation*}
    where $v_i \in \vrt(P) \setminus \{ 0 \}$ and $w_j \in \vrt(Q) \setminus \{ 0 \}$.
    Let $F = \conv(v_1,\dots,v_s)$ and $G = \conv(w_1, \dots, w_r)$.
    Since each $v_i \neq 0$ and since $0$ is a vertex of $P$ we have $0 \not\in F$.
    Similarly $0 \not\in G$.
    Now $\varphi(F*G) = H$.
    It remains to show that $F \in L(P)$ and $G \in L(Q)$.
    
    Since $H \in L(P \vee Q)$ we find some $\gamma \in \mathbb{R}^{m+n}$ and $c \in \mathbb{R}$ such that $\vrt(H)$ maximizes $\gamma(x,y) \le c$ for $P \vee Q$.
    We can write $\gamma(x,y) = \alpha(x) + \beta(y)$ for some $\alpha \in (\mathbb{R}^m)^*$ and $\beta \in (\mathbb{R}^n)^*$.
    One easily shows that $\vrt(F)$ maximizes $\alpha(x) \le c$ for $P$ and that $\vrt(G)$ maximizes $\beta(y) \le c$ for $Q$.
    Thus $F \in L(P)$ and $G \in L(Q)$.
\end{proof}

\begin{lemma}\label{lemma:if_0_not_in_F_and_0_not_in_G_then_F*G_in_L(PvQ)}
    If $F \in L(P)$ and $G \in L(Q)$ with $0 \not\in F$ and $0 \not\in G$, then
    \begin{equation*}
        \varphi(F*G) \in L(P \vee Q)
    \end{equation*}
    where $0 \not\in \varphi(F*G)$.
\end{lemma}

\begin{proof}
    Let $F \in L(P)$ and $G \in L(Q)$ where $0 \not\in F$ and $0 \not\in G$.
    Write
    \begin{equation*}
        \vrt(P) = \{ v_1,\dots,v_s\} \sqcup \{v_{s+1},\dots,v_k,0 \}
    \end{equation*}
    where $\vrt(F) = \{ v_1, \dots, v_s \}$, and write
    \begin{equation*}
        \vrt(Q) = \{ w_1,\dots,w_{r}\} \sqcup \{w_{r+1},\dots,w_\ell,0 \} 
    \end{equation*}
    where $\vrt(G) = \{ w_1, \dots, w_r \}$.
    Let 
    \begin{equation*}
        H = \conv((v_1,0),\dots,(v_s,0),(0,w_1),\dots,(0,w_r)).
    \end{equation*}
    Now $\varphi(F*G) = H$.
    It remains to show that $H \in L(P\vee Q)$.
    
    Since $F \in L(P)$ we find some $\alpha \in (\mathbb{R}^m)^*$ and $a \in \mathbb{R}$ such that $\vrt(F)$ maximizes $\alpha(x) \le a$ for $P$.
    Since $0 \not\in F$ and $0 \in \vrt(P)$ we need to have $a \neq 0$.
    Similarly, we find $\beta \in (\mathbb{R}^n)^*$ and $b \in \mathbb{R}$ such that $\vrt(G)$ maximizes $\beta(y) \le b$ for $Q$ where $b \neq 0$.
    Define $\gamma \in (\mathbb{R}^{m+n})^*$ by
    \begin{equation*}
        \gamma(x,y) = \frac{\alpha(x)}{a} + \frac{\beta(y)}{b}.
    \end{equation*}
    One easily shows that $\vrt(H)$ maximizes $\gamma(x,y) \le 1$ for $P \vee Q$.
    Thus $H \in L(P \vee Q)$.
    Since $\gamma(0,0) < 1$, we have $0 \not \in H$.
\end{proof}
We are now ready to give a full description of the faces of $P \vee Q$.
The following proposition is essentially Proposition 2.3 from \cite{McMullen}.

\begin{proposition}\label{prop:faces_of_PvQ}
    The maps
    \begin{align*}
        \{ F \in L(P) \mid 0 \in F \} \times \{ G \in L(Q) \mid 0 \in G \} &\longrightarrow \{ H \in L(P \vee Q) \mid 0 \in H \} \\
        (F,G) &\longmapsto F \vee G
    \end{align*}
    and
    \begin{align*}
        \{ F \in L(P) \mid 0 \not\in F \} \times \{ G \in L(Q) \mid 0 \not\in G \} &\longrightarrow \{ H \in L(P \vee Q) \mid 0 \not\in H \} \\
        (F,G) &\longmapsto \varphi(F * Q)
    \end{align*}
    are bijections.
\end{proposition}

\begin{proof}
    Let us look at the first map.
    By Lemma \ref{lemma:if_0_in_F_and_0_in_G_then_FvG_in_L(PvQ)} the map is well-defined.
    By Lemma~\ref{lemma:if_0_in_H_then_H=FvG} the map is surjective.
    We show injectivity.
    Suppose
    \begin{equation*}
        F = \conv(V), \quad F' = \conv(V'), \quad G=\conv(W), \quad G'=\conv(W'),
    \end{equation*}
    for some $V,V' \subseteq \vrt(P)$ and $W,W' \subseteq \vrt(Q)$, where \mbox{$0 \in V,V'$} and $0 \in W,W'$.
    Assume $F \vee G = F' \vee G'$.
    Let $v \in V$.
    Now
    \begin{equation*}
        (v,0) \in F \vee G = F' \vee G' = \conv(V' \times \{ 0 \}^n \cup \{ 0 \}^m \times W').
    \end{equation*}
    From this we obtain
    \begin{equation*}
        v = \sum_{v' \in V' \setminus \{ 0 \}} t_{v'} v', \quad t_{v'} \ge 0, \quad \sum_{v' \in V' \setminus \{ 0 \}} t_{v'} \le 1.
    \end{equation*}
    If $v = 0$ then $v \in V'$.
    Assume $v \neq 0$.
    If we had $v \not\in V'$ then we could write $v$ as a convex combination
    \begin{equation*}
        v = \sum_{v' \in V' \setminus \{ 0 \}} t_{v'} v' + \left( 1 - \sum_{v' \in V' \setminus \{ 0 \}} t_{v'} \right)0
    \end{equation*}
    which would be a contradiction as now $v$ is in the convex hull of other vertices of $P$.
    Thus in both cases $v \in V'$.
    Therefore $V \subseteq V'$.
    By symmetry, $V' \subseteq V$ and hence $F=F'$.
    One can show $W=W'$ in exactly the same way.
    Therefore $(F,G) = (F',G')$ and hence the first map is injective.
    
    Then we look at the second map.
    By Lemma \ref{lemma:if_0_not_in_F_and_0_not_in_G_then_F*G_in_L(PvQ)} the map is well-defined.
    By Lemma~\ref{lemma:if_0_not_in_H_then_H_isom_F*G} the map is surjective.
    It remains to show injectivity.
    Let $F,F',G,G'$ be as above but this time $0 \not\in V,V'$ and $0 \not\in W,W'$.
    Suppose $\varphi(F*G) = \varphi(F'*G')$.
    Let $v \in V$.
    Now
    \begin{equation*}
        (v,0) \in \varphi(F*G) = \varphi(F'*G') = \conv(V' \times \{ 0 \}^n \cup \{ 0 \}^m \times W').
    \end{equation*}
    From this we obtain
    \begin{equation*}
        v = \sum_{v' \in V'} t_{v'}v', \quad t_{v'} \ge 0, \quad \sum_{v' \in V'} t_{v'} \le 1.
    \end{equation*}
    Similarly as before, if we had $v \not\in V'$ then we could write $v$ as a convex combination
    \begin{equation*}
        v = \sum_{v' \in V'} t_{v'}v' + \left( 1 - \sum_{v'\in V'}t_{v'} \right)0
    \end{equation*}
    and hence $v$ would be a convex combination of other vertices of $P$.
    Thus we must have $v \in V'$ and therefore $V \subseteq V'$.
    By symmetry $V' \subseteq V$.
    One shows $W=W'$ in the same way. 
    Thus $(F,G) = (F',G')$.
    Hence the second map is also injective.
\end{proof}

In Proposition 2.3 in his paper \cite{McMullen}, McMullen describes the faces of $P \vee Q \subseteq \mathbb{R}^{m+n}$ as being either of the form $F \vee G$ with $0 \in F$ and $0 \in G$, or of the form $F*G$ where $0 \not\in F$ and $0 \not\in G$.
The statement of Proposition \ref{prop:faces_of_PvQ} above is essentially this with the difference that one has to project the faces $F*G \subseteq \mathbb{R}^{m+n+1}$ down with the projection $\varphi$ in order to get faces in the right ambient space.
This projection turns out to be an injection (and hence provides an isomorphism) on the faces of the form $F*G$ with $0 \not \in F$ and $0 \not \in G$.

\begin{lemma}\label{lemma:projection_injective_on_R*S}
    Let $R \subseteq \mathbb{R}^m$ and $S \subseteq \mathbb{R}^n$ be any polytopes such that $0 \not \in \aff(R)$ and \mbox{$0 \not \in \aff(S)$}.
    Then $\varphi$ is injective on $\aff(R*S)$.
    In particular, $\varphi$ is injective on $R*S$ and therefore $R*S \cong \varphi(R*S)$.
\end{lemma}

\begin{proof}
    Let $a_0,\dots,a_d \in \mathbb{R}^m$ be an affine basis for $\aff(R)$ and let $b_0,\dots,b_e \in \mathbb{R}^n$ be an affine basis for $\aff(S)$.
    Since $0 \not\in \aff(R) = \aff(a_0,\dots,a_d)$, the points $a_0,\dots,a_d,0$ are affinely independent.
    Similarly $b_0,\dots,b_e,0$ are affinely independent.
    Now
    \begin{equation*}
        (a_0,0,0),\dots,(a_d,0,0),(0,b_0,1),\dots,(0,b_e,1) \in \mathbb{R}^{m+n+1}
    \end{equation*}
    is an affine basis for $\aff(R*S)$.
    Suppose $x,y \in \aff(R*S)$ are such that $\varphi(x) = \varphi(y)$.
    We now have affine combinations
    \begin{align*}
        x &= \sum_{i=0}^d t_i (a_i,0,0) + \sum_{j=0}^e s_j(0,b_j,1)\\
        y&= \sum_{i = 0}^d t_i'(a_i,0,0) + \sum_{j=0}^e s_j'(0,b_j,1).
    \end{align*}
    Since $\varphi(x) = \varphi(y)$ we get
    \begin{equation*}
        \sum_{i=0}^d t_i(a_i,0) + \sum_{j=0}^e s_j(0,b_j) = \sum_{i=0}^d t_i'(a_i,0) + \sum_{j=0}^e s_j'(0,b_j).
    \end{equation*}
    We thus have affine combinations
    \begin{align*}
        \sum_{i=0}^d t_ia_i + (1 - t_0 - \dots - t_d)0  &= \sum_{i=0}^d t_i'a_i + (1 - t_0' - \dots - t_d')0 \\
        \sum_{j=0}^e s_jb_j + (1-s_0-\dots - s_e)0 &= \sum_{j=0}^e s_j' b_j + (1 - s_0' - \dots - s_e')0.
    \end{align*}
    Since $a_0,\dots,a_d,0$ are affinely independent, the first equation implies $t_i = t_i'$ for all $i = 0,\dots,d$.
    Since $b_0,\dots,b_e,0$ are affinely independent, the second equation implies $s_j = s_j'$ for all $j = 0 ,\dots, e$.
    Hence $x=y$.
    Thus $\varphi$ is injective on $\aff(R*S)$.
\end{proof}

From Lemma \ref{lemma:projection_injective_on_R*S} it follows that if $F \in L(P)$ and $G \in L(Q)$ are such that $0 \not \in F$ and $0 \not\in G$, then $F*G \cong \varphi(F*G)$.

\section{Order and chain polytopes from ordinal sums}\label{section:O(A<B)_and_C(A<B)}
In this section, we quickly look at how the order and chain polytopes behave under ordinal sums.
Let $A$ and $B$ be posets on disjoint sets. 
At the beginning of Section \ref{section:PvQ} we saw that $\mathcal{C}(A < B) = \mathcal{C}(A) \vee \mathcal{C}(B)$.
We state this as a separate result for future reference.

\begin{proposition}\label{prop:C(A<B)=C(A)VC(B)}
    If $A$ and $B$ are posets on disjoint sets, then
    \begin{equation*}
        \mathcal{C}(A < B) = \mathcal{C}(A) \vee \mathcal{C}(B).
    \end{equation*}
\end{proposition}

Almost as easily we obtain a similar result for the order polytope.

\begin{proposition}\label{prop:O(A<B)=O(A)VO(Bop)}
    If $A$ and $B$ are posets on disjoint sets, then
    \begin{equation*}
        \mathcal{O}(A < B) \cong \mathcal{O}(A) \vee \mathcal{O}(B^{\op}).
    \end{equation*}
\end{proposition}

\begin{proof}
    Since 
    \begin{equation*}
        \{ \text{filters of } A < B \} = \{ F \cup B \mid F \text{ filter in } A \} \cup \{ \text{filters in } B \}
    \end{equation*}
    we have
    \begin{equation*}
        \vrt(\mathcal{O}(A < B)) = \vrt(\mathcal{O}(A)) \times \{ 1 \}^B \cup \vrt(\mathcal{O}(B)) \times \{ 0 \}^A.
    \end{equation*}
    Thus 
    \begin{equation*}
        \mathcal{O}(A < B) = \conv(\mathcal{O}(A) \times \{ 1 \}^B \cup \mathcal{O}(B) \times \{ 0 \}^A).
    \end{equation*}
    Recall that by definition
    \begin{equation*}
        \mathcal{O}(A) \vee \mathcal{O}(B^{\op}) = \conv(\mathcal{O}(A) \times \{ 0 \}^B \cup \mathcal{O}(B^{\op}) \times \{ 0 \}^A).
    \end{equation*}
    Define $f \colon \mathbb{R}^{A \cup B} \to \mathbb{R}^{A \cup B}$ by
    \begin{equation*}
        f(x)_i = 
        \begin{cases}
            x_i & i \in A \\
            1-x_i & i \in B
        \end{cases}
    \end{equation*}
    for all $x \in \mathbb{R}^{A \cup B}$ and all $i \in A \cup B$.
    Here $f$ is a bijective affine map.
    In particular, $f$ is injective on $\mathcal{O}(A < B)$ so it remains to show that it maps $\mathcal{O}(A < B)$ onto $\mathcal{O}(A) \vee \mathcal{O}(B^{\op})$.
    Recall that the map
    \begin{align*}
        \mathbb{R}^B &\longrightarrow \mathbb{R}^B  \\
        x &\longmapsto \bone - x
    \end{align*}
    gives an isomorphism $\mathcal{O}(B) \to \mathcal{O}(B^{\op})$.
    We therefore have
    \begin{align*}
        f(\mathcal{O}(A) \times \{ 1 \}^B) &= \mathcal{O}(A) \times \{ 0 \}^B \\ f(\mathcal{O}(B) \times \{ 0 \}^A) &= \mathcal{O}(B^{\op}) \times \{ 0 \}^A.
    \end{align*}
    Since $f$ is affine, it preserves convex combinations.
    Thus
    \begin{align*}
        f(\mathcal{O}(A < B)) &= f(\conv(\mathcal{O}(A) \times \{ 1 \}^B \cup \mathcal{O}(B) \times \{ 0 \}^A)) \\
        &= \conv(f(\mathcal{O}(A) \times \{ 1 \}^B \cup \mathcal{O}(B) \times \{ 0 \}^A )) \\
        &= \conv(\mathcal{O}(A) \times \{ 0 \}^B \cup \mathcal{O}(B^{\op}) \times \{ 0 \}^A) \\
        &= \mathcal{O}(A) \vee \mathcal{O}(B^{\op}).
    \end{align*}
\end{proof}

Since $\mathcal{O}(B) \cong \mathcal{O}(B^{\op})$, it is clear that the recursive formulas in Propositions~\ref{prop:C(A<B)=C(A)VC(B)} and \ref{prop:O(A<B)=O(A)VO(Bop)} are very similar.
However, note that the isomorphism in Proposition \ref{prop:O(A<B)=O(A)VO(Bop)} does not preserve the origin.
Since the subdirect sum $\vee$ is defined with respect to prescribed vertex~0, it is thus possible (and usual), that $\mathcal{O}(A<B) \not\cong \mathcal{C}(A<B)$, even when we would have $\mathcal{O}(A) \cong \mathcal{C}(A)$ and $\mathcal{O}(B) \cong \mathcal{C}(B)$.
For example, let $A<B$ be the $X$-poset from Figure~\ref{fig:Xposet} where $A$ is the two element antichain consisting of the minimal elements, and $B$ is the 3~element V-shaped poset consisting of the rest of the elements.

\section{$f$-Polynomials}\label{section:f_polynomials}
In this section, we look at $f$-polynomials of polytopes.
We start by computing the $f$-polynomial of the polytope $P \vee Q$.

\begin{proposition}\label{prop:f_poly_of_PvQ}
    Let $P$ and $Q$ be polytopes both containing the origin as a vertex.
    The $f$-polynomial of $P \vee Q$ is given by
    \begin{equation*}
        f_{P \vee Q} = \frac{1}{x} f_P^0 f_Q^0 + f_P^1f_Q^1.
    \end{equation*}
\end{proposition}

\begin{proof}
    We have $f_{P\vee Q} = f_{P \vee Q}^0 + f_{P \vee Q}^1$.
    Using the first part of Proposition \ref{prop:faces_of_PvQ} and the fact that
    \begin{equation*}
        \dim(F \vee G) = \dim(F) + \dim(G)        
    \end{equation*}
    we obtain
    \begin{align*}
        f_{P \vee Q}^0 &= \sum_{\substack{H \in L(P \vee Q) \\ 0 \in H}} x^{\dim(H) + 1} \\
        &= \sum_{\substack{F \in L(P); \  0 \in F \\ G \in L(Q); \ 0 \in G}} x^{\dim(F \vee Q) + 1} \\ 
        &= \frac{1}{x} \sum_{\substack{F \in L(P); \ 0 \in F \\ G \in L(Q); \ 0 \in G}} x^{\dim(F) + 1} x^{\dim(G) + 1} \\
        &= \frac{1}{x} f_P^0 f_Q^0.
    \end{align*}
    Using the second part of Proposition \ref{prop:faces_of_PvQ} and the remarks after Lemma \ref{lemma:projection_injective_on_R*S}, together with the fact that
    \begin{equation*}
        \dim(F*G) = \dim(F) + \dim(G) + 1,
    \end{equation*}
    we get
    \begin{align*}
        f_{P \vee Q}^1 &= \sum_{\substack{H \in L(P \vee Q) \\ 0 \not \in H} } x^{\dim(H) + 1} \\
        &= \sum_{\substack{F \in L(P); \  0 \not\in F \\ G \in L(Q); \ 0 \not\in G}} x^{\dim(\varphi(F*G)) + 1}  \\
        &= \sum_{\substack{F \in L(P); \  0 \not\in F \\ G \in L(Q); \ 0 \not\in G}} x^{\dim(F*G) + 1}\\
        &= \sum_{\substack{F \in L(P); \ 0 \not \in F \\ G \in L(Q); \ 0 \not\in G}} x^{\dim(F) + 1}x^{\dim(G) + 1} \\
        &= f_P^1f_Q^1. 
    \end{align*}
\end{proof}

Combining Proposition \ref{prop:f_poly_of_PvQ} with Propositions \ref{prop:C(A<B)=C(A)VC(B)} and \ref{prop:O(A<B)=O(A)VO(Bop)} we get the following Corollary.

\begin{corollary}\label{corollary:f_poly_of_O(A<B)_and_C(A<B)}
    For any posets $A$ and $B$ on disjoint sets,
    \begin{equation*}
        f_{\mathcal{O}(A < B)} = \frac{1}{x} f_{\mathcal{O}(A)}^0 f_{\mathcal{O}(B^{\op})}^0 +  f_{\mathcal{O}(A)}^1  f_{\mathcal{O}(B^{\op})}^1
    \end{equation*}
    and
    \begin{equation*}
        f_{\mathcal{C}(A<B)} = \frac{1}{x} f^0_{\mathcal{C}(A)} f_{\mathcal{C}(B)}^0 +  f_{\mathcal{C}(A)}^1  f_{\mathcal{C}(B)}^1.
    \end{equation*}
\end{corollary}

In both the order and the chain polytope the origin is always a vertex, but in the chain polytope the origin turns out to be a simple vertex.
With this observation we obtain the following inequality.

\begin{lemma}\label{lemma:f_C(P)_less_f_O(P)}
    For any poset $\mathcal{P}$,
    \begin{equation*}
        f_{\mathcal{C}(\mathcal{P})}^0 \le f_{\mathcal{O}(\mathcal{P})}^0.
    \end{equation*}
\end{lemma}

\begin{proof}
    The empty set and each singleton of $\mathcal{P}$ are antichains of $\mathcal{P}$.
    Therefore the origin $0 = e_\emptyset$ and each standard basis vector $e_i \in \mathbb{R}^{\mathcal{P}}$, where  $i \in \mathcal{P}$, are vertices of $\mathcal{C}(\mathcal{P})$.
    Because $\mathcal{C}(\mathcal{P})$ is contained in the hypercube, we see that the edges of $\mathcal{C}(\mathcal{P})$ containing the origin are exactly the line segments between $0$ and $e_i$, where $i \in \mathcal{P}$.
    Thus $0$ is contained in exactly $\#\mathcal{P}$ edges of $\mathcal{C}(\mathcal{P})$ and hence the vertex figure $\mathcal{C}(\mathcal{P}) / 0$ is a $(\#\mathcal{P} - 1)$-simplex.
    
    Since the empty set is a filter of $\mathcal{P}$, the origin is also a vertex in $\mathcal{O}(\mathcal{P})$.
    Thus the vertex figure $\mathcal{O}(\mathcal{P})/0$ is a $(\# \mathcal{P} - 1)$-dimensional polytope.
    Since simplicies have the smallest $f$-vectors for their dimension, it follows that
    \begin{equation*}
        f_{\mathcal{C}(\mathcal{P})}^0 = x f_{\mathcal{C}(\mathcal{P})/0} \le x f_{\mathcal{O}(\mathcal{P})/0} = f_{\mathcal{O}(\mathcal{P})}^0.
    \end{equation*}
\end{proof}

In this paper, we pay special attention to whether a face contains or does not contain the origin as a vertex.
We therefore want to compare the number of these kinds of faces to each other.
The origin plays no special role in the next result so we state it in a slightly more general setting.

\begin{proposition}\label{prop:estimation}
    Let $P$ be any polytope and let $v$ be a vertex of $P$.
    Then for all $k \ge 0$, 
    \begin{equation*}
        \# \{ F \in L(P) \mid v \in F, \ \dim(F)= k \} \le \# \{ G \in L(P) \mid v \not\in G, \ \dim(G) = k-1\}.
    \end{equation*}
\end{proposition}

\begin{proof}
    Identify each face with its vertex set.
    Define a function
    \begin{equation*}
        \varphi \colon \{ F \in L(P) \mid v \in F \} \longrightarrow \{ G \in L(P) \mid v \not\in G \}
    \end{equation*}
    by defining $\varphi(F)$ to be an arbitrary maximal subset of $F \setminus \{ v \}$ that is a face of $P$.
    Note that by maximality of $\varphi(F)$ we have $\dim(\varphi(F)) = \dim(F) - 1$.
    Once we have shown that $\varphi$ is injective, the result follows.
    We claim that
    \begin{equation*}
        \varphi(F) \vee \{ v \} = F 
    \end{equation*}
    for any $F \in L(P)$ with $v \in F$, where $\vee$ is the join in the lattice $L(P)$.
    Since $\varphi(F) \cup \{ v \} \subseteq F$ we have $\varphi(F) \vee \{ v \} \subseteq F$.
    Now since
    \begin{equation*}
        \varphi(F) \subsetneqq \varphi(F) \vee \{ v \} \subseteq F
    \end{equation*}
    we get
    \begin{equation*}
        \dim(F) - 1 = \dim(\varphi(F)) < \dim(\varphi(F) \vee \{ v \}) \le \dim(F).
    \end{equation*}
    Therefore $\dim(\varphi(F) \vee \{ v \}) = \dim(F)$ and thus $\varphi(F) \vee \{ v \} = F$.
    Hence $\varphi$ has a left inverse and therefore it is injective.
\end{proof}

We therefore obtain the following, which shall be useful later.

\begin{corollary}\label{corollary:f_P^0_<_xf_P^1}
    For any polytope $P$ containing the origin as a vertex,
    \begin{equation*}
        f_P^0 \le x f_P^1.
    \end{equation*}
\end{corollary}

\begin{proof}
    Multiplication by $x$ shifts the coefficients up one degree and both $f_P^0$ and $xf_P^1$ have constant term zero.
    The inequality is thus equivalent with Proposition~\ref{prop:estimation}.
\end{proof}

We also get the following result, which might be of independent interest.

\begin{proposition}\label{lemma:f_ineq_pyrvee_vs_join}
    If $P$ and $Q$ are any polytopes both having origin as a vertex, then
    \begin{equation*}
        f_{\pyr(P \vee Q)} \ge f_{P*Q}.
    \end{equation*}
\end{proposition}
\begin{proof}
    Since for any polytopes $R$ and $S$ it holds that $f_{R*S} = f_Rf_S$ (see e.g.\ Section 15.1.3 in \cite{handbook}) and since a single point has $f$-polynomial $1+x$, we get from Proposition~ \ref{prop:f_poly_of_PvQ} that 
    \begin{equation*}
        f_{\pyr(P \vee Q)} = \frac{1}{x} f_P^0 f_Q^0 + f_P^1f_Q^1 + f_P^0f_Q^0 + xf_P^1f_Q^1.
    \end{equation*}
    The desired inequality $f_{\pyr(P \vee Q)} \ge f_{P*Q}$ is thus equivalent with 
    \begin{equation*}
        \frac{1}{x} f_P^0 f_Q^0 + f_P^1f_Q^1 + f_P^0f_Q^0 + xf_P^1f_Q^1 \ge (f_P^1 + f_P^0)(f_Q^1 + f_Q^0).
    \end{equation*}
    After multiplying both sides with $x$, expanding the product on the right-hand side and subtracting common terms, this is seen to be equivalent with
    \begin{equation*}
        f_P^0f_Q^0 + x^2f_P^1f_Q^1 \ge xf_P^1f_Q^0 + xf_P^0f_Q^1.
    \end{equation*}
    Define $g_P \coloneqq xf_P^1- f_P^0$ and  $g_Q \coloneqq xf_Q^1 - f_Q^0$.
    By Corollary \ref{corollary:f_P^0_<_xf_P^1} we have $g_P  \ge 0$ and $g_Q \ge 0$.
    By substituting $xf_P^1 = f_P^0 + g_P$ and $xf_Q^1 = f_Q^0 + g_Q$ into the above equation, we get the equivalent inequality
    \begin{equation*}
        f_P^0f_Q^0 + (f_P^0+g_P)(f_Q^0+g_Q) \ge (f_P^0 + g_P)f_Q^0 + f_P^0(f_Q^0+g_Q).
    \end{equation*}
    By expanding products and subtracting common terms, this inequality is seen to be equivalent with $g_Pg_Q \ge 0$ which certainly holds as $g_P \ge 0$ and $g_Q \ge 0$.
    The proposition follows.
\end{proof}

In our main theorem, we will construct new posets by taking disjoint unions of earlier constructed posets.
It is easily verified that if $\mathcal{P}$ and $\mathcal{Q}$ are posets on disjoint sets, then $\mathcal{O}(\mathcal{P} \sqcup \mathcal{Q}) = \mathcal{O}(\mathcal{P}) \times \mathcal{O}(\mathcal{Q})$ and $\mathcal{C}(\mathcal{P} \sqcup \mathcal{Q}) = \mathcal{C}(\mathcal{P}) \times \mathcal{C}(\mathcal{Q})$.
Recall that if $P$ and $Q$ are any polytopes, then the map
\begin{align*}
    (L(P) \setminus \{ \emptyset \}) \times (L(Q) \setminus \{ \emptyset \}) & \longrightarrow L(P \times Q) \setminus \{ \emptyset \} \\
    (F,G) &\longmapsto F \times G
\end{align*}
is a bijection.
Hence the $f$-polynomial of $P \times Q$ is given by
\begin{align*}
    f_{P \times Q} &= 1 + \sum_{\substack{F \times G \in L(P \times Q) \\ \dim(F \times G) \ge 0}} x^{\dim(F \times G) + 1} \\
    &= 1 + \frac{1}{x} \sum_{\substack{F \in L(P); \ \dim(F) \ge 0 \\ G \in L(Q); \  \dim(G) \ge 0}} x^{\dim(F) + 1} x^{\dim(G) + 1} \\
    &= 1 + \frac{(f_P - 1)(f_Q - 1)}{x}
\end{align*}
where we used the fact that $\dim(F \times G) = \dim(F) + \dim(G)$ for all faces $F$ and $G$.
We therefore easily verify the following lemma.
\begin{lemma}\label{lemma_f_ineq_times}
    Let $P_1,P_2,Q_1,Q_2$ be any polytopes such that $f_{P_1} \le f_{P_2}$ and $f_{Q _1} \le f_{Q_2}$.
    Then $f_{P_1 \times Q_1} \le f_{P_2 \times Q_2}$.
\end{lemma}

\section{Main theorem}\label{section:main_theorem}
To prove our main theorem we need one more lemma.
For any poset $\mathcal{P}$ let
\begin{equation*}
    \alpha_\mathcal{P} \coloneqq \frac{1}{x} f^0_{\mathcal{C}(\mathcal{P})} \quad 
    \beta_\mathcal{P} \coloneqq \frac{1}{x} f^0_{\mathcal{O}(\mathcal{P})} \quad 
    \gamma_\mathcal{P} \coloneqq f_{\mathcal{O}(\mathcal{P})}^1 \quad 
    \delta_\mathcal{P} \coloneqq  f_{\mathcal{C}(\mathcal{P})}^1.
\end{equation*}
Notice that since $\mathcal{C}(\mathcal{P}) = \mathcal{C}(\mathcal{P}^{\op})$, we have $\alpha_\mathcal{P} = \alpha_{\mathcal{P}^{\op}}$ and $\delta_\mathcal{P} = \delta_{\mathcal{P}^{\op}}$.

\begin{lemma}\label{lemma:alpha_beta_gamma_delta}
    For any poset $\mathcal{P}$, if $f_{\mathcal{O}(\mathcal{P})} \le f_{\mathcal{C}(\mathcal{P})}$ then 
    \begin{enumerate}
        \item $x(\beta_\mathcal{P} - \alpha_\mathcal{P}) \le \delta_\mathcal{P} - \gamma_\mathcal{P}$, and 
        \item $\alpha_\mathcal{P} \le \beta_\mathcal{P} \le \gamma_\mathcal{P} \le \delta_\mathcal{P}.$
    \end{enumerate}
\end{lemma}

\begin{proof}
    Since $f_{\mathcal{O}(\mathcal{P})} \le f_{\mathcal{C}(\mathcal{P})}$ we get $        f_{\mathcal{O}(\mathcal{P})}^0 +  f_{\mathcal{O}(\mathcal{P})}^1 \le f_{\mathcal{C}(\mathcal{P})}^0 +  f_{\mathcal{C}(\mathcal{P})}^1$
    and thus
    \begin{equation*}
        f_{\mathcal{O}(\mathcal{P})}^0 - f_{\mathcal{C}(\mathcal{P})}^0 \le  f_{\mathcal{C}(\mathcal{P})}^1 - f_{\mathcal{O}(\mathcal{P})}^1.
    \end{equation*}
    Looking at the definitions we obtain the desired inequality in part one.
    
    The first inequality in the second part follows from Lemma \ref{lemma:f_C(P)_less_f_O(P)}.
    The second inequality follows from Corollary \ref{corollary:f_P^0_<_xf_P^1}.
    For the last inequality we notice that since $\alpha_\mathcal{P} \le \beta_\mathcal{P}$ we get by the first part that $0 \le x(\beta_\mathcal{P} - \alpha_\mathcal{P}) \le \delta_\mathcal{P} - \gamma_\mathcal{P}$.
\end{proof}

We are now ready to state and prove our main theorem.
If a poset $\mathcal{P}$ does not have the $X$-poset from Figure \ref{fig:Xposet} as a subposet, we say $\mathcal{P}$ is $X$-\emph{free}.

\begin{theorem}\label{theorem:main_theorem}
    Let $\mathcal{F}$ be the family of posets built inductively by starting with all $X$-free posets and using the constructions $(\mathcal{P},\mathcal{Q}) \leadsto \mathcal{P} < \mathcal{Q}$ and $(\mathcal{P},\mathcal{Q}) \leadsto \mathcal{P} \sqcup \mathcal{Q}$.
    Then for any poset $\mathcal{P}$ in this family $\mathcal{F}$, it holds that
    \begin{equation*}
        f_{\mathcal{O}(\mathcal{P})} \le f_{\mathcal{C}(\mathcal{P})}.
    \end{equation*}
\end{theorem}

\begin{proof}
    If a poset $\mathcal{P}$ does not have the $X$-poset as a subposet, then by the result of Hibi and Li (\cite{Hibi&Li}, Corollary 2.3) we have equality $f_{\mathcal{O}(\mathcal{P})} = f_{\mathcal{C}(\mathcal{P})}$.
    
    Suppose $A$ and $B$ are two posets (on disjoint sets) such that $f_{\mathcal{O}(A)} \le f_{\mathcal{C}(A)}$ and $~{f_{\mathcal{O}(B)} \le f_{\mathcal{C}(B)}}$.
    From Lemma \ref{lemma_f_ineq_times} we get
    \begin{equation*}
        f_{\mathcal{O}(A \sqcup B)} = f_{\mathcal{O}(A) \times \mathcal{O}(B)} \le f_{\mathcal{C}(A) \times \mathcal{C}(B)} = f_{\mathcal{C}(A \sqcup B)}.
    \end{equation*}
    It remains to show $f_{\mathcal{O}(A < B)} \le f_{\mathcal{C}(A < B)}$.
    Note that we can apply Lemma \ref{lemma:alpha_beta_gamma_delta} to $A$ and $B$ but also to $B^{\op}$ since
    \begin{equation*}
        f_{\mathcal{O}(B^{\op})} = f_{\mathcal{O}(B)} \le f_{\mathcal{C}(B)} = f_{\mathcal{C}(B^{\op})}
    \end{equation*}
    as $\mathcal{O}(B) \cong \mathcal{O}(B^{\op})$ and $\mathcal{C}(B) = \mathcal{C}(B^{\op})$.
    By Corollary \ref{corollary:f_poly_of_O(A<B)_and_C(A<B)} we can write
    \begin{align*}
        f_{\mathcal{O}(A < B)} &= \frac{1}{x} f_{\mathcal{O}(A)}^0 f_{\mathcal{O}(B^{\op})}^0 +  f_{\mathcal{O}(A)}^1  f_{\mathcal{O}(B^{\op})}^1 \\
        &= x\beta_A \beta_{B^{\op}} + \gamma_A \gamma_{B^{\op}}
    \end{align*}
    and 
    \begin{align*}
        f_{\mathcal{C}(A<B)} &= \frac{1}{x} f^0_{\mathcal{C}(A)} f_{\mathcal{C}(B)}^0 +  f_{\mathcal{C}(A)}^1  f_{\mathcal{C}(B)}^1 \\
        &= x\alpha_A \alpha_B + \delta_A \delta_B.
    \end{align*}
    Showing $f_{\mathcal{O}(A<B)} \le f_{\mathcal{C}(A < B)}$ is therefore equivalent to showing
    \begin{equation*}
        x\beta_A \beta_{B^{\op}} - x \alpha_A \alpha_B \le \delta_A\delta_B - \gamma_A \gamma_{B^{\op}}.
    \end{equation*}
    By rewriting both sides here, we see that our aim is thus to show
    \begin{equation*}
        x\alpha_A(\beta_{B^{\op}} - \alpha_B) + x(\beta_A - \alpha_A)\beta_{B^{\op}} \le \gamma_A(\delta_B - \gamma_{B^{\op}}) + (\delta_A - \gamma_A)\delta_B. \tag{1}
    \end{equation*}
    If we apply part 1 of Lemma \ref{lemma:alpha_beta_gamma_delta} to the poset $B^{\op}$ and part 2 to the poset $A$, we get
    \begin{equation*}
        \begin{cases}
            0 \le x(\beta_{B^{\op}} - \alpha_{B^{\op}}) \le \delta_{B^{\op}} - \gamma_{B^{\op}} \\
            0 \le \alpha_A \le \gamma_A.
        \end{cases}
    \end{equation*}
    Multiplying these together yields
    \begin{equation*}
        x\alpha_A(\beta_{B^{\op}} - \alpha_B)= x\alpha_A(\beta_{B^{\op}} - \alpha_{B^{\op}}) \le \gamma_A(\delta_{B^{\op}} - \gamma_{B^{\op}}) = \gamma_A(\delta_B - \gamma_{B^{\op}}). \tag{2}
    \end{equation*}
    Applying part 1 of Lemma \ref{lemma:alpha_beta_gamma_delta} to the poset $A$ and part 2 to the poset $B^{\op}$, we obtain
    \begin{equation*}
        \begin{cases}
            0 \le x(\beta_A - \alpha_A) \le \delta_A - \gamma_A \\
            0 \le \beta_{B^{\op}} \le \delta_{B^{\op}} = \delta_B.
        \end{cases}
    \end{equation*}
    Multiplying these inequalities together we get
    \begin{equation*}
        x(\beta_A - \alpha_A)\beta_{B^{\op}} \le (\delta_A - \gamma_A)\delta_B. \tag{3}
    \end{equation*}
    Adding (2) and (3) together gives us the inequality (1).
    We therefore have
    \begin{equation*}
        f_{\mathcal{O}(A < B)} \le f_{\mathcal{C}(A < B)}
    \end{equation*}
    finishing the proof.
    
\end{proof}

\begin{figure}[t]
    \centering
    \includegraphics[scale=1]{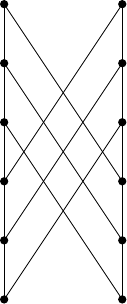}
    \caption{An $X$-free poset which is not constructed from posets of height at most 2 and taking disjoint unions and ordinal sums.}
    \label{fig:zigzag}
\end{figure}

The main result in \cite{Ahmad-Fourier-Joswig} is obtained from our Theorem \ref{theorem:main_theorem} by starting with antichains and using ordinal sums only.
For Theorem \ref{theorem:main_theorem}, a more natural inductive family of posets may be starting from posets of height at most 2 and taking ordinal sums and disjoint unions.
The family of posets thus constructed is clearly a subset of the posets constructed in Theorem \ref{theorem:main_theorem}, as posets of height at most~2 cannot have the $X$-poset as a subposet.
However, these two classes of posets are not equal since, for example, the poset in Figure \ref{fig:zigzag} (and its obvious generalizations) is $X$-free, but cannot be constructed from height 2 posets with disjoint unions and ordinal sums.

\printbibliography
    
\end{document}